\documentclass[12pt,reqno]{amsart}
\usepackage[utf8]{inputenc}
\usepackage[english]{babel}
\usepackage{amsmath}
\usepackage{amsfonts}
\usepackage{amssymb}
\usepackage{amsthm}
\usepackage{siunitx}
\usepackage{bbm}
\usepackage[colorlinks,allcolors=blue]{hyperref}
\usepackage[square,numbers]{natbib}

\usepackage{comment}
\usepackage{tasks}
\usepackage{enumitem}
\usepackage{stix}
\usepackage{multicol}
\usepackage{mathtools}
\usepackage{graphicx}
\usepackage[margin=1in]{geometry}
\usepackage{mathrsfs}
\usepackage{pgf,tikz,pgfplots}
\newtheorem{theorem}{Theorem}
\newtheorem{lemma}[theorem]{Lemma}
\newtheorem{proposition}[theorem]{Proposition}
\newtheorem{definition}[theorem]{Definition}

\newtheorem{remark}[theorem]{Remark}

\usetikzlibrary{calc}
\usetikzlibrary{arrows}
\pgfplotsset{compat=1.15}
\newenvironment{nouppercase}{%
  \renewcommand{\uppercasenonmath}[1]{}}{}

\title{Thermodynamics Formalism for a Class of Hyperbolic Transcendental Meromorphic Functions}
\author{Hamid Naderiyan}
\address{Department of Mathematics, University of California -- Riverside}
\email{hnaderiy@ucr.edu}

\begin{document}

\begin{abstract}
    This paper studies the thermodynamics formalism in the context of complex dynamics. We establish the thermodynamics formalism for the class of hyperbolic transcendental meromorphic functions of B--class  where the poles have bounded multiplicities, the Nevanlinna order is finite, and infinity is not an asymptotic value. We showed the existence and uniqueness of the conformal measure and the invariant Gibbs measure that is equivalent to the conformal measure.
    
\end{abstract}

\begin{nouppercase}
\maketitle
\end{nouppercase}

\tableofcontents

\section{Introduction}

\noindent The field of thermodynamic formalism has a rich history in physics and mathematics. There have been extensive studies on this field not only in various setting, but also its relation with other fields such multifractal analysis, dimension theory, and etc. , see \citep[section 1.1]{Barr} for a short overview.

\noindent The motivation of the current work is a question posed by M. Urba\'{n}ski on the thermodynamic formalism for the transcendental meromorphic functions of B--class where the poles have bounded multiplicities and infinity is not an asymptotic value under the hyperbolicity condition. This class of functions gained attention at least recently by a work of W. Bergweiler \& J. Kotus \cite{Bergweiler-Kotus} where they found a sharp estimate for the Hausdorff dimension of the escaping set of such functions. We call this class of functions the BK--class, see definition \ref{BK} for the precise definition. In fact, for a function $f$ of BK--class where all poles except finitely many have multiplicities bounded by $M$ and $\rho$ is the Nevanlinna order of $f$, it is shown by W. Bergweiler \& J. Kotus that $\text{HD}(I_f)\leq \frac{2M\rho }{2+M\rho}$, where $I_f$ is the escaping set of $f$, i.e., the set of all $z\in \mathbb{C}$ for which $|f^n(z)| \to \infty$ as $n\to \infty$, see \cite{Bergweiler-Kotus}. Note that the Julia set is the boundary of the escaping set as shown in \cite{Dominq}. 

\noindent In this work, we adapt the approach of M. Urba\'{n}ski and V. Mayer in \cite{Urbanski-Mayer} by which they established thermodynamic formalism for a class of transcendental meromorphic function subject to a growth rate for the derivative. Many family of functions are subject to this growth rate condition such as exponential functions, sine functions and elliptic functions, for more details see  \cite[p. 7]{Urbanski-Mayer}. It is important to note that the assumed growth condition implies that the function is hyperbolic in the sense that the derivative of the iterate of function grows exponentially all over the Julia set.  M. Urba\'{n}ski and V. Mayer in \cite{Urbanski-Mayer} first obtain a conformal measure for a tame potential that can be perceived as a perturbation of the geometric potential. Then they use it to construct a Gibbs measure that is also the equilibrium state. They also show the Bowen's formula that the Hausdorff dimension of the radial Julia set is equal to the zero of the pressure function. 

\noindent The context of conformal dynamics is another topic of research in dynamical systems and ergodic theory that can be traced back to the work of Patterson \cite{Pattersonn} where he studied the limit set of Fuchsian groups, and by D. Sullivan \cite{Sullivan} where he studied a broader class of discrete groups. We also refer the readers to \cite{Stratmann} for more details about the conformal measure and its relation with Gibbs measure.

\noindent The method of this paper for obtaining a conformal measure is the analysis of the transfer operator. First, we consider a potential function (equation (\ref{potential function})) associated to a transcendental meromorphic function $f$ of BK--class (definition \ref{BK}) which is expressed in terms of a the derivative of $f$ with respect to an appropriate Riemannian metric (equation (\ref{metric derivative})). Then we assign a transfer operator for this potential (equation (\ref{transfer operator})). One major step is establishing the boundedness of this operator (proposition \ref{boundedness of transfer operator}). The derivative of the function $f \in$ BK at $z$ has an estimate (inequality \ref{derivative growth rate}) in terms of $f(z)$ and $z$ as long as the value is large enough. This estimate along Rippon-Stallard estimate (inequality (\ref{Rippon-Stallard})) under the assumption of hyperbolicity give rise to nice estimate all over the Julia set of $f$ (lemma \ref{estimate for tau derivative stronger}). This estimate helps to show that the transfer operator is bounded. Another ingredient of this fact is the Borel's theorem  in Nevanlinna theory (estimate (\ref{Borel theorem})). Later, we construct a family of measure that satisfies a tightness property (lemma \ref{tighetness}). Applying Prokhorov's theorem we obtain the conformal measure for the function $f$, the central result of this paper (proposition \ref{conformal measure existence}).  By taking integral of some fixed point of the transfer operator against this conformal measure we obtain an invariant Gibbs measure that is equivalent to the conformal measure (equation (\ref{invariant measure})). The fact that the escaping set of $f$ has measure $0$ (lemma \ref{limit inf bounded}) is used to show the uniqueness of the Gibbs measure that is also ergodic with support on the radial Julia set which is the main result of this work, see theorem \ref{main theorem}.

\section{Preliminaries}
Throughout this paper we make conventions that 
\begin{itemize}
    \item $c, r, t$ are positive real constants,
    \item $f$ is a transcendental meromorphic function whose poles have multiplicity at most $M=M_f$,
    \item $\tau$ is a parameter that is usually subject to $1<\tau<1+\frac{1}{M}$,
    \item $M_u$ is the upper bound of a series given in (\ref{Borel theorem}),
    \item $\rho=\rho(f)$ is the Nevanlinna order of $f$,
    \item $\mathcal{J}(f)$ is the Julia set of $f$ minus $\infty$,
    \item $K$ is a real number greater than $1$,
    \item B--class is the  meromorphic Eremenko-Lyubich class,
    \item BK--class is the Bergweiler-Kotus class (see definition \ref{BK}),
    \item $T=T_f$ is a positive number that usually satisfies dist$(\mathcal{J}(f), 0)\geq T$,
    \item $\delta=(1/4)$ dist$(\mathcal{J}(f),\overline{\mathcal{P}(f)})$
    \item $\mathcal{L}_t$ is the transfer operator associated to the potential function $\Phi_t$,
    \item $\varphi$ is a bounded continuous function on $\mathcal{J}(f)$,
    \item $\mathbb{1}_B$ is the characteristic function of the set $B$,
    \item $A(x) \asymp B(x)$ stands for $c A(x) \leq B(x) \leq CA(x)$, where $0<c<C$ are constants independent of $x$.
\end{itemize}

\noindent The Fatou set $\mathcal{F}(f)$ of a (non-linear) meromorphic function $f: \mathbb{C} \to \hat{\mathbb{C}}=\mathbb{C}\cup\{\infty\}$ is defined to be the set of all $z \in \mathbb{C}$ for which all the iterates $f^n$ of $f$ form a normal family in some neighborhood of $z$. The Julia set is defined to be $\hat{\mathcal{J}}(f)= \hat{\mathbb{C}}\setminus\mathcal{F}(f)$. Also, we say $a$ is an asymptotic value of $f$ if there exists a continuous curve $\lambda(t)$ such that $\lambda(t) \to \infty$ and $f(\lambda(t))\to a$ as $t \to \infty$. The set of singularities of $f^{-1}$ coincides with the set consisting of critical values or asymptotic values of $f$. We denote the finite singularities of $f^{-1}$ by sing($f^{-1}$).  The post-critical set $\mathcal{P}(f)$ of a function $f$ consists of singularities of $f^{-1}, f^{-2}, f^{-3}, ...$. We develop almost all of our theory in this paper over the set $\mathcal{J}(f):=\hat{\mathcal{J}}(f) \setminus \{\infty\}$ where we still call it the Julia set. Also, the escaping set of $f$ is defined to be
$$I(f)=\{z\in \mathbb{C} \; : \; |f^n(z)| \to \infty \text{ \; as\; } n\to \infty\}.$$
Additionally, we define the radial Julia set of $f$ to be $\mathcal{J}(f) \setminus I(f)$. The hyperbolic dimension of $f$ is defined to be the Hausdorff dimension of the radial Julia set. We say a meromorphic transcendental function $f$ is in Eremenko-Lyubich class if sing$(f^{-1})$ is bounded and we write $f \in $ B. For more detailed explanations, we refer to \citep[ch. 13]{MeromorphicDynamicsII}. 

\noindent The Nevanlinna order of $f$ is given by
$$\rho=\rho(f)=\lim_{r \to \infty}\dfrac{\log T(r,f)}{\log r}$$
see \citep[p. 13]{Urbanski-Mayer} for the definition of the Nevanlinna characteristic $ T(r,f)$. A theorem of Borel for a meromorphic function $f$ states that the following series 
$$\sum_{\substack{f(z)=w \\ z\neq 0}}\dfrac{1}{|z|^{u}},$$
converges when $u>\rho$ and diverges when $u<\rho$, except for possibly two (Picard's) exceptional points $w$ as stated in \citep[97]{Tsuji}. Also, under the assumption dist($0,f^{-1}(B))>0$, the above series converges uniformly in $w$, i.e. when $u>\rho$, there exits $M_u>0$ such that 
\begin{equation}\label{Borel theorem}
    \sum_{\substack{f(z)=w \\ z\neq 0}}\dfrac{1}{|z|^{u}} \leq M_u, \; \; \; \; w \in B,
\end{equation}
shown in \citep[p. 16]{Urbanski-Mayer}. 

\noindent We are ready to assert the BK--class (W. Bergweiler \& J. Kotus) functions central to this paper. 
\begin{definition}[\textbf{$\mathbf{BK}-$class}]\label{BK}
    We say a transcendental meromorphic function $f$ is in class Bergweiler-Kotus if all the following conditions hold,
    \begin{itemize}
        \item[i)] $f \in$ B, i.e., $f$ is of Eremenko–Lyubich class,
        \item[ii)] the Nevanlinna order $\rho$ is finite,
        \item[iii)] $\infty$ is not an asymptotic value,
        \item[iv)] there exists $M \in \mathbb{N}$ such that the multiplicity of all poles is at most $M$.
    \end{itemize}
    In this case, we write $f\in $ BK.
\end{definition}
\noindent We make an important remark that there is no consensus on the definition of hyperbolic transcendental meromorphic function. We know that the expanding property of a rational function $f$
\begin{equation}
    |(f^n)'(z)|\geq c K^n, \hspace{1cm} c>0, \;\; K>1, \; \; z \in \mathcal{J}(f), \; \; n \in \mathbb{N},
\end{equation}
is equivalent to 
\begin{equation}\label{disjoint hyp}
  \mathcal{J}(f)\cap \overline{\mathcal{P}(f)}=\emptyset,  
\end{equation}
while this is not true for transcendental meromorphic functions. In this work, we focus on functions with condition (\ref{disjoint hyp}).
\begin{definition}[\textbf{Hyperbolic Function}]\label{toplogically hyp}
    We say a meromorphic function $f: \mathbb{C} \to \hat{\mathbb{C}}$ is hyperbolic if it satisfies condition (\ref{disjoint hyp}).
\end{definition}

\noindent Note that Picard's exceptional points are in the Fatou set for a topologically hyperbolic function since they are asymptotic values. Additionally, we want to emphasize that for a hyperbolic transcendental mermorphic function of B--class we have dist$(\mathcal{J}(f), \overline{\mathcal{P}(f)})>0$ by a result of Rippon and Stallard, see theorem B in \citep{stallard}. This gap between the Julia set and the closure of the post-critical set enhance the work with the inverse branches in the sense that for any $w \in \mathcal{J}(f)$ and $x,y$ close enough to $w$, then $f_z^{-n}(x)$ and $f_z^{-n}(y)$ make sense for all $n \in \mathbb{N}, z \in f^{-n}(\{w\})$. In fact, by setting 
\begin{equation}\label{delta}
    \delta:=\frac{1}{4}\textup{dist}(\mathcal{J}(f),\overline{\mathcal{P}}_f)
\end{equation}
for every $w \in \mathcal{J}(f)$ and every $n \in \mathbb{N}$ the map 
$$f^n : f^{-n}\big( D(w,2\delta) \big) \to D(w,2\delta)$$ 
is a covering map, so there is a countable family of disjoint open sets $\{ O_z\}_z$ such that 
$$f^{-n}\big( D(w,2\delta) \big)=\cup_{z} O_z$$
where 
$$f^{-n}: D(w,2\delta) \to O_z, \hspace{1cm} \{z\}=f^{-n}\big( D(w,2\delta) \big) \cap O_z$$
is analytic for each $z$. Thus, we obtain a countable family of inverse functions all being analytic,
$$\{f_z^{-n} \hspace{2mm}: \hspace{2mm} f_z^{-n}: D(w,2\delta) \to O_z, \hspace{3mm}  f_z^{-n}(w)=z  \}.$$
\begin{remark}\label{f_z=f_z'}
    Note that $f_z^{-n}=f_{z'}^{-n}$ on $D(w,2\delta) \cap D(w',2\delta)$ by the identity theorem.
\end{remark}

\noindent We also consider the Riemannian metrics
$$ds=\frac{|dz|}{1+|z|^{\tau}}, \; \; \; \; \tau>0,$$
on the complex plane $\mathbb{C}$. A straightforward calculation for a complex function $f: \mathbb{C} \to \hat{\mathbb{C}}$  yields that the norm of the ds-derivative (derivative with respect to ds metric) of $f$ is given by
$$\| f' \|_{\tau}(z)=|f'(z)|\dfrac{1+|z|^{\tau}}{1+|f(z)|^{\tau}}.$$
A simple calculation gives the following lemma.

\begin{lemma}\label{distortion for T}
   For every $T, \tau>0$, there exits $K_{T,\tau}>1$ such that for any complex function $f$ and any complex number $z$ with $|z|, |f(z)| \geq T>0$ we have,
    $$K_{T,\tau}^{-1} \dfrac{|z|^{\tau}}{|f(z)|^{\tau}} \leq \dfrac{1+|z|^{\tau}}{1+|f(z)|^{\tau}} \leq K_{T,\tau}\dfrac{|z|^{\tau}}{|f(z)|^{\tau}}.$$
\end{lemma}
\noindent If $f$ is a topologically hyperbolic function, there exists $c \not\in \mathcal{J}(f)$ and so $0 \not\in \mathcal{J}_{gfg^{-1}}=g(\mathcal{J}(f))$ for $g(z)=z+c$. Therefore, we can work with topologically hyperbolic B--class functions where $0 \not \in \mathcal{J}(f)$ without loss of generality. This means, there exists $T>0$ such that,
\begin{equation}\label{T}
    |z|, |f(z)| \geq T>0, \hspace{.5cm} z \in\mathcal{J}(f)
\end{equation}
and so instead of the actual $\tau$-norm of the derivative given above, we work with a simpler form below,
\begin{equation}\label{metric derivative}
    |f'(z)|_{\tau}:=|f'(z)|\dfrac{|z|^{\tau}}{|f(z)|^{\tau}}.
\end{equation}
More precisely, the potential functions (given in (\ref{potential function})) associated with $|.|_{\tau}$ and $\|.\|_{\tau}$ differ only by a constant for a fixed $t$, hence the thermodynamic formalisms for both of them are equivalent as shown in \citep[p. 444]{Munday}.

\noindent For every $R>0$ we set $B(R)=\{z: |z|>R \}\cup \{\infty\}$ and 
$D(z,R)=\{ y \in \mathbb{C}: |y-z|<R\}$. We know for $f \in $ BK, there exists $R_0>0$ such that $f^{-1}(B(R_0))=\cup_{j=1}^{\infty} U_j,$
where each $U_i$ is a bounded simply connected region containing exactly one of the poles of $f$, say $a_j$. Without loss of generality, we can assume further
$$1\leq|a_1|\leq|a_2|\leq...,$$
as well as $R_0>1$ and $|f(0)|<R_0$. Also, we have
$$f\sim \left( \frac{b_j}{z-a_j} \right)^{m_j}, \; \; \; z \to a_j,$$
\begin{equation}\label{derivative near pole}
    |f'|\sim  \frac{1}{|b_j|}|f|^{1+\frac{1}{m_j}}, \; \; \; z \to a_j,
\end{equation}
where $m_j$ is the multiplicity of the pole $a_j$ and $b_j$ is a complex number such that
$$|b_j|\leq 4R_0|a_j|.$$
The details can be found in \citep[p. 5374]{Bergweiler-Kotus}. Combining this inequality with the estimate (\ref{derivative near pole}) results in the following essential growth rate of the derivative, also known as the rapid growth rate in \citep[p. 5]{Urbanski-Mayer},
\begin{equation}\label{derivative growth rate}
|f'(z)|\geq c_0\frac{1}{|z|}|f(z)|^{1+\frac{1}{m_j}}\geq c_0\frac{1}{|z|}|f(z)|^{1+\frac{1}{M}}, \; \; \; z \to a_j,
\end{equation}
where $c_0$ only depends on $R_0$. In addition to this, there is a general estimate for the derivative of a topologically hyperbolic B--class function $f$ obtained by Rippon and Stallard in \cite[p. 3252]{stallard},
    \begin{equation}\label{Rippon-Stallard}
        |(f^n)'(z)|>cK^n\dfrac{|f^n(z)|+1}{|z|+1}
    \end{equation}
    where $c>0$, $K>1$ and $z$ is any point of the Julia set except where $f^n$ is not analytic. Now this estimate combined with lemma \ref{distortion for T} gives the following estimate suitable for our setting,
    \begin{equation}\label{Rippon-Stallard-modified}
        |(f^n)'(z)|>c_1K_1^n\dfrac{|f^n(z)|}{|z|},
    \end{equation}
    where $c_1>0$, $K_1>1$ and $z$ is any point of the Julia set except where $f^n$ is not analytic. Before the following lemma, we assume $t>0$ is a real number throughout this paper. The following lemma gives an estimate for the $\tau-$norm derivative of $f$ for a general topologically hyperbolic function.
\begin{lemma}\label{estimate for tau derivative}
    There exists $c>0$ such that for every $w \in \mathcal{J}(f)$ and every $z\in f^{-1}(\{w\})$ we have,
        $$\left|(f_z^{-1})'(w)\right|_{\tau}^{t}=|f'(z)|_{\tau}^{-t} \leq c^{t}\dfrac{|w|^{-(1-\tau)t}}{|z|^{(\tau-1)t}}.$$
\end{lemma}
\begin{proof}
   By the Rippon-Stallard expansion estimate (\ref{Rippon-Stallard-modified}) for $n=1$ we find,
    $$|f'(z)|_{\tau}^{-t}\leq c_1^{-t} K_1^{-t} \dfrac{|f(z)|^{-t}}{|z|^{-t}}\dfrac{|z|^{-\tau t}}{|f(z)|^{-\tau t}}=c_1^{-t} K_1^{-t} \dfrac{|f(z)|^{-(1-\tau)t}}{|z|^{(\tau -1)t}}=c_1^{-t} K_1^{-t} \dfrac{|w|^{-(1-\tau)t}}{|z|^{(\tau -1)t}}.$$
\end{proof}
\noindent However, if $f$ is also in BK--class, we can obtain a better estimate for the $\tau-$norm derivative of $f$.
\begin{lemma}\label{estimate for tau derivative stronger}
    There exists $c>0$ such that for every $w \in \mathcal{J}(f)$ and every $z\in f^{-1}(\{w\})$ we have,
    $$\left|(f_z^{-1})'(w)\right|_{\tau}^{t}=|f'(z)|_{\tau}^{-t} \leq c^{t}\dfrac{|w|^{-(1+\frac{1}{M}-\tau)t}}{|z|^{(\tau-1)t}}.$$
\end{lemma}
\begin{proof}
    Given $w\in \mathcal{J}(f)\cap B(R_0)$, the inequality (\ref{derivative growth rate}) gives the following estimate,
    $$|f'(z)|_{\tau}^{-t}\leq c_0^{-t}\dfrac{|f(z)|^{-(1+\frac{1}{M})t}}{|z|^{-t}}\dfrac{|z|^{-\tau t}}{|f(z)|^{-\tau t}} = c_0^{-t}\dfrac{|f(z)|^{-(1+\frac{1}{M}-\tau)t}}{|z|^{(\tau-1)t}} =c_0^{-t}\dfrac{|w|^{-(1+\frac{1}{M}-\tau)t}}{|z|^{(\tau-1)t}}.$$
    Next, if $w\in \mathcal{J}(f)$ and $|w|\leq R_0$, we apply lemma \ref{estimate for tau derivative} to write,
    \begin{align*}
        \left|(f_z^{-1})'(w)\right|_{\tau}^{t}=|f'(z)|_{\tau}^{-t} \leq c^{t}\dfrac{|w|^{-(1-\tau)t}}{|z|^{(\tau-1)t}}& = c^t|w|^{\frac{1}{M}t}\dfrac{|w|^{-(1+\frac{1}{M}-\tau)t}}{|z|^{(\tau-1)t}}\\
        &\leq (cR_0^{\frac{1}{M}})^t\dfrac{|w|^{-(1+\frac{1}{M}-\tau)t}}{|z|^{(\tau-1)t}}.
    \end{align*}
    Now if we set $c_1=\max \left(c_0^{-1}, cR_0^{\frac{1}{M}} \right)$, for every $w \in \mathcal{J}(f)$ we find 
    $$\left|(f_z^{-1})'(w)\right|_{\tau}^{t}=|f'(z)|_{\tau}^{-t} \leq c_1^{t}\dfrac{|w|^{-(1+\frac{1}{M}-\tau)t}}{|z|^{(\tau-1)t}}$$
\end{proof}
\begin{definition}[\textbf{Gibbs State}]
    We say a Borel probability measure $\mu$ is a Gibbs state for $f$ when there exist real constants $P$ and $K>1$ such that for every $z \in \mathcal{J}(f)$ and every $n \in \mathbb{N}$, 
    $$K^{-1}\leq \dfrac{\mu \left(f_z^{-n}\left(D(f^n(z),\delta) \right) \right)}{\exp \left(S_n\Phi_t(z)-nP \right)} \leq K.$$
\end{definition}
\begin{definition}[\textbf{Conformal Measure}]
    We say a Borel probability measure $m$ on $\mathcal{J}(f)$ is $\psi$-conformal if 
$$m(f(A))=\int_A\psi dm,$$
for every Borel set $A \subseteq \mathcal{J}(f)$ such that $f$ is injective on $A$, where $\psi: \mathcal{J}(f) \to \mathbb{R}$.
\end{definition}
\noindent The main targets of this paper is establishing the existence and uniqueness of the Gibbs state and conformal measure. First, we establish the existence of a conformal measure. Then an integral against this conformal measure gives an invariant Gibbs state for $f$. Eventually we argue that these two measures are equivalent and unique.
\section{Geometric Potential and Transfer Operator Properties}
\noindent We define the geometric potential for each $t>0$,
$$\Phi_t: \mathcal{J}(f) \to \mathbb{R},$$
\begin{equation}\label{potential function}
    \Phi_t(z):=t\log |f'(z)|^{-1}_{\tau}=t\log \big|(f_z^{-1})'\big(f(z)\big)\big|_{\tau};
\end{equation}
and the transfer operator,
$$\mathcal{L}_t: C_b(\mathcal{J}(f),\mathbb{C}) \to C_b(\mathcal{J}(f),\mathbb{C}),$$
\begin{equation}\label{transfer operator}
    \mathcal{L}_t \varphi(w):=\sum_{f(z)=w}\exp(\Phi_t(z))\varphi(z)=\sum_{f(z)=w}|f'(z)|_{\tau}^{-t}\varphi(z)=\sum_{f(z)=w} \big|(f_z^{-1})'(w)\big|_{\tau}^t\varphi(z),
\end{equation}

\noindent where $C_b(\mathcal{J}(f),\mathbb{C})$ is the space of complex-valued bounded continuous functions on $\mathcal{J}(f)$ and $\varphi \in C_b(\mathcal{J}(f),\mathbb{C})$. We show below $\mathcal{L}_t$ is a bounded operator under certain conditions. Note that $\mathbb{1}_X$ denotes the characteristic function of $X$, where $X \subseteq \mathcal{J}(f)$. When $X=\mathcal{J}(f)$ we drop the subscript and we simply write $\mathbb{1}$ instead. Also, we denote the supremum norm by $\| . \|$, i.e. $\| \varphi \|=\sup_{z \in \mathcal{J}(f)} |\varphi(z)|$. 

\begin{lemma}\label{estimate for transfer operator}
    There exists $c>0$ such that for every $w \in \mathcal{J}(f)$,
         $$|\mathcal{L}_t \varphi (w)|\leq c^t|w|^{-(1+\frac{1}{M}-\tau)t}\|\varphi\|\sum_{f(z)=w}\dfrac{1}{|z|^{(\tau-1)t}}.$$
\end{lemma}
\begin{proof}
   Lemma \ref{estimate for tau derivative stronger} gives the following estimate,
    \begin{align*}
        |\mathcal{L}_t \varphi(w)|=\left|\sum_{f(z)=w}|f'(z)|_{\tau}^{-t}\varphi(z)\right| & \leq c^{t}|w|^{-(1+\frac{1}{M}-\tau)t}\|\varphi\|\sum_{f(z)=w}\dfrac{1}{|z|^{(\tau-1)t}}.
    \end{align*}
\end{proof}
\noindent Before proving the following proposition we remind the readers that the upper bound $M_u$ was given in (\ref{Borel theorem}). It is clear that by our assumption that the Julia set is bounded away from the origin such $M_u$ exists on the Julia set. In fact, when $t> \frac{\rho}{\tau -1}$ and $\tau>1$, for every $w \in \mathcal{J}(f)$ we have,
\begin{equation}\label{Borel theorem 2}
    \sum_{f(z)=w}\dfrac{1}{|z|^{(\tau-1)t}} \leq M_{(\tau-1)t}.
\end{equation}
\begin{proposition}\label{boundedness of transfer operator}
     The transfer operator $\mathcal{L}_t$ is a bounded operator on $C_b(\mathcal{J}(f),\mathbb{C})$ equipped with the supremum norm, when $ 1<\tau<1+\frac{1}{M}$ and $t> \frac{\rho}{\tau -1}$. In fact, there exists $c>0$ such that 
     $$\|\mathcal{L}_t\| \leq c^tM_{(\tau-1)t}.$$
\end{proposition}
\begin{proof}
    Due to the fact that $|w| \geq T>0$ for all $w \in \mathcal{J}(f)$ and lemma \ref{estimate for transfer operator}, there exists $c>0$ such that,
    \begin{align*}
        | \mathcal{L}_t\varphi(w)|\leq c^t|w|^{-(1+\frac{1}{M}-\tau)t}\|\varphi\|\sum_{f(z)=w}\dfrac{1}{|z|^{(\tau-1)t}} \leq c^tT^{-(1+\frac{1}{M}-\tau)t}\|\varphi\|M_{(\tau-1)t}
    \end{align*}
    Thus, we obtain
    $$ \|\mathcal{L}_t\| \leq (cT^{-(1+\frac{1}{M}-\tau)})^tM_{(\tau-1)t}. $$
    \end{proof}
\noindent    Following the above proposition we assume $ 1<\tau<1+\frac{1}{M}$ and $t> \frac{\rho}{\tau -1}$ from now on, unless stated otherwise.
 
  \noindent  For each $n \in \mathbb{N}$, the $n^{th}$ iterate of the transfer operator takes the form,
    $$\mathcal{L}_t^n: C_b(\mathcal{J}(f),\mathbb{C}) \to C_b(\mathcal{J}(f),\mathbb{C})$$
\begin{equation}
    \mathcal{L}_t^n \varphi(w)=\sum_{f^n(z)=w}\exp(S_n\Phi_t(z))\varphi(z)=\sum_{f^n(z)=w}|(f^n)'(z)|_{\tau}^{-t}\varphi(z)=\sum_{f^n(z)=w}\big|(f_z^{-n})'(w)\big|_{\tau}^t\varphi(z),
\end{equation}
where 
\begin{align*}
    S_n\Phi_t(z) & =\sum_{i=0}^{n-1}\Phi_t(f^{i}(z))=\Phi_t(z)+\Phi_t(f(z))+...+\Phi_t(f^{n-1}(z))\\
    & = -t\log|f'(z)|_{\tau}-t\log|f'(f(z))|_{\tau}-t\log|f'(f^2(z))|_{\tau}-...-t\log|f'(f^{n-1}(z))|_{\tau}\\
    & = -t\log\left| (f^n)'(z) \right|_{\tau}=t\log\left| (f_z^{-n})'(f^n(z)) \right|_{\tau}
\end{align*}
is the ergodic sum of $\Phi_t$. Next, we show a series of lemma.

\begin{lemma}\label{Koebe f'}
    There exists $K>1$ such that for every $n \in \mathbb{N}$, every $w\in \mathcal{J}(f)$, every $w'$ with $|w'-w| < \delta$ and every $z \in f^{-n}(\{ w\})$, we have
   $$\left|\left|(f^{-n}_z)'(w)\right| - \left|(f^{-n}_z)'(w')\right|\right|\leq K\left|(f^{-n}_z)'(w)\right||w-w'|.$$
\end{lemma}
\begin{proof}
    This follows from Koebe's distortion theorem \citep[p. 287]{MeromorphicDynamics}.
\end{proof}
\begin{lemma}\label{Koebe f}
    There exists $K>1$ such that for every $n \in \mathbb{N}$, every $w_1,w_2 \in \mathcal{J}(f)$ with $|w_1-w_2| < \delta$ and every $z \in f^{-n}(\{ w_1\})$, we have
    $$\left|(f_z^{-n})(w_1)-(f_z^{-n})(w_2)\right|\leq K|(f_z^{-n})(w_1)||w_1-w_2|.$$
\end{lemma}
\begin{proof}
    We use the mean value theorem and lemma \ref{Koebe f'} to find,
    $$\left|(f_z^{-n})(w_1)-(f_z^{-n})(w_2)\right|\leq (1+K\delta)|(f_z^{-n})'(w_1)||w_1-w_2|.$$
    Then Rippon-Stallard expansion estimate (\ref{Rippon-Stallard-modified}) yields,
    \begin{align*}
        \left|(f_z^{-n})(w_1)-(f_z^{-n})(w_2)\right|&\leq (1+K\delta)c_1^{-1}K_1^{-n}\dfrac{|(f_z^{-n})(w_1)|}{|w_1|} |w_1-w_2|\\
        & \leq (1+K\delta)c_1^{-1}K_1^{-n} T^{-1}|(f_z^{-n})(w_1)||w_1-w_2|.
    \end{align*}
\end{proof}
    \begin{lemma}\label{Ionescu S_n}
        There exists $K>1$ such that for every $n \in \mathbb{N}$, every $w_1,w_2 \in \mathcal{J}(f)$ with $|w_1-w_2|<\delta$ and  
        every $z \in f^{-n}(\{w_1 \})$, we have
        \begin{align*}
            \big| S_n\Phi_t\left(f^{-n}_z(w_1)\right) -S_n\Phi_t\left(f^{-n}_z(w_2)\right) \big| &= \left|t\log\left|(f_z^{-n})'(w_1)\right|_{\tau} - t\log\left|(f_z^{-n})'(w_2)\right|_{\tau}\right|\\
            & \leq tK|w_1-w_2|.
        \end{align*}
    \end{lemma}
    \begin{proof}
    We use lemmas \ref{Koebe f'} and \ref{Koebe f} with the corresponding constants $K_2,K_3>1$ and the fact that $\log(x+1)\leq x$ for $x>-1$ to write,
    \begin{align*}
      &  \left|t\log\left|(f_z^{-n})'(w_1)\right|_{\tau} - t\log\left|(f_z^{-n})'(w_2)\right|_{\tau}\right| \\
       = & \left|t\log \left| \dfrac{(f_z^{-n})'(w_1)}{(f_z^{-n})'(w_2)} \right|+t\tau\log\left|\dfrac{w_1}{w_2} \right|-t\tau\log \left| \dfrac{(f_z^{-n})(w_1)}{(f_z^{-n})(w_2)} \right|\right| \\
       \leq & \;t\log \left(1+K_2|w_1-w_2| \right) +t\tau \log \left(1+\dfrac{|w_1-w_2|}{T} \right) +t\tau \log \left(1+K_3|w_1-w_2| \right)\\
       \leq &\; tK_2|w_1-w_2| +t\tau T^{-1}|w_1-w_2|+t\tau K_3|w_1-w_2| \\
       \leq &\; t(K_2+\tau T^{-1}+\tau K_3)|w_1-w_2|.
    \end{align*}
    \end{proof}
    \begin{lemma}\label{exp Ionescu S_n}
        There exists $K>1$ such that for every $n \in \mathbb{N}$, every $w_1,w_2 \in \mathcal{J}(f)$ with $|w_1-w_2|<\delta$ and  
        every $z \in f^{-n}(\{w_1 \})$, we have
        \begin{align*}
            \left| \exp\left(S_n\Phi_t\left(f^{-n}_z(w_1)\right)\right) -\exp\left(S_n\Phi_t\left(f^{-n}_z(w_2)\right)\right) \right| &= \left|\left|(f_z^{-n})'(w_1)\right|_{\tau}^t - \left|(f_z^{-n})'(w_2)\right|_{\tau}^t\right|\\
            & \leq te^{tK}K\left|(f_z^{-n})'(w_1)\right|_{\tau}^t|w_1-w_2|.
        \end{align*}
    \end{lemma}
    \begin{proof}
        It is enough to use lemma \ref{Ionescu S_n} along with the fact $|e^x-e^y|\leq e^{|x-y|}e^x|x-y|$.        
    \end{proof}
\begin{lemma}\label{distortion of transfer operator 0}
 There exists $K>1$ such that for every $n \in \mathbb{N}$ and every $w_1,w_2 \in \mathcal{J}(f)$ with $|w_1-w_2|<\delta$ we have
    $$\left| \mathcal{L}_t^n\mathbb{1}(w_1) - \mathcal{L}_t^n\mathbb{1}(w_2) \right| \leq te^{tK}K \mathcal{L}_t^n\mathbb{1}(w_1)|w_1-w_2| $$
\end{lemma}
\begin{proof}
With regards to remark \ref{f_z=f_z'} and lemma \ref{exp Ionescu S_n} we can write,
       \begin{align*}
           \big| \mathcal{L}_t^n\mathbb{1}(w_1) - \mathcal{L}_t^n\mathbb{1}(w_2) \big| &= \left| \sum_{f^n(z)=w_1} \left|(f_z^{-n})'(w_1)\right|_{\tau}^t - \sum_{f^n(z)=w_2} \left|(f_z^{-n})'(w_2)\right|_{\tau}^t\right| \\
           & = \left| \sum_{f^n(z)=w_1} \left(\left|(f_z^{-n})'(w_1)\right|_{\tau}^t - \left|(f_z^{-n})'(w_2)\right|_{\tau}^t\right) \right| \\
           & \leq \sum_{f^n(z)=w_1} \left| \left|(f_z^{-n})'(w_1)\right|_{\tau}^t - \left|(f_z^{-n})'(w_2)\right|_{\tau}^t \right| \\
           & \leq \sum_{f^n(z)=w_1} te^{tK}K\left|(f_z^{-n})'(w_1)\right|_{\tau}^t|w_1-w_2| \\
           & =  te^{tK}K\mathcal{L}_t^n\mathbb{1}(w_1)|w_1-w_2|.
        \end{align*}
\end{proof}
\noindent It is clear from this lemma that for every $n \in \mathbb{N}$ and every $w_1, w_2 \in \mathcal{J}(f)$ with $|w_1-w_2|< \delta$,
$$\mathcal{L}_t^n\mathbb{1}(w_2) \leq (1+te^{tK}K\delta)\mathcal{L}_t^n\mathbb{1}(w_1).$$
\noindent This shows the map,
$$P_t: \;\mathcal{J}(f) \rightarrow \mathbb{R}$$
$$w \mapsto \limsup_{n \to \infty}\dfrac{1}{n}\log \mathcal{L}_t^n\mathbb{1}(w),$$
is locally constant. In fact, this map is constant by the following lemma.

\begin{lemma}\label{distortion of transfer operator}
    There exist $c>0$ such that for every $R>0$, there exist $K_R>1$ such that for every $n\in \mathbb{N}$ and every $w_1, w_2 \in \mathcal{J}(f)\cap D(0, R)$, we have
   $$\mathcal{L}_t^n\mathbb{1}(w_2) \leq c K_R^t\mathcal{L}_t^n\mathbb{1}(w_1).$$
\end{lemma}
\begin{proof}
    First, we mention the following blow-up property of the Julia set shown in \citep[p. 18]{Urbanski-Mayer}. Given $R>0$, there exists $N \in \mathbb{N}$ such that for all $w_2 \in \mathcal{J}(f)\cap D(0, R)$, we have $\mathcal{J}(f)\cap D(0, R) \subseteq f^N\left( D(w_2,\delta) \right)$. Since all Picard's exceptional points are asymptotic values and therefore they lie in the Fatou set of $f$, one can find $w_3 \in D(w_2, \delta)$ such that $w_1=f^N(w_3)$. Since $\mathcal{J}(f)$ is bounded away from the origin and $f^N$ is analytic on $\mathcal{J}(f)$, there exists $K_R$ such that $|(f^N)'(w_3)|^t_{\tau}\leq K_R^t$ for all $w_3 \in \mathcal{J}(f)\cap D(0, R+\delta)$. Hence, we use lemma \ref{distortion of transfer operator} to find,
    \begin{align*}
        \mathcal{L}_t^n\mathbb{1}(w_2) &\leq (1+te^{tK}K\delta)\mathcal{L}_t^n\mathbb{1}(w_3) \\
        & = (1+te^{tK}K\delta) \sum_{f^n(z)=w_3} \left|(f^{n})'(z)\right|_{\tau}^{-t}\\
        & = (1+te^{tK}K\delta) \sum_{f^{n}(z)=w_3} \left|(f^{n+N})'(z)\right|_{\tau}^{-t}\left|(f^{N})'(f^n(z))\right|_{\tau}^{t} \\
        & \leq (1+te^{tK}K\delta) \sum_{f^{n}(z)=w_3} \left|(f^{n+N})'(z)\right|_{\tau}^{-t}K_R^t\\
        & \leq (1+te^{tK}K\delta)K_R^t \sum_{f^{n+N}(z)=w_1} \left|(f^{n+N})'(z)\right|_{\tau}^{-t}\\
        & = (1+te^{tK}K\delta)K_R^t \mathcal{L}_t^n\mathbb{1}(w_1)
    \end{align*}
\end{proof}
\section{Construction of Conformal Measure}
\begin{definition}\label{pressure}
    The pressure of the potential $\Phi_t$ is defined by 
    $$P_t= \limsup_{n \to \infty}\dfrac{1}{n}\log \mathcal{L}_t^n\mathbb{1}(w), \hspace{1cm} w \in \mathcal{J}(f).$$
\end{definition}

\noindent We emphasize that for every $n \in \mathbb{N}$, the adjoint operator $\left(\mathcal{L}_t^n \right)^*$ takes every finite measure on $\mathcal{J}(f)$ to another finite measure. In fact, for every finite (complex) measure $m$, every $g \in C_b(\mathcal{J}(f),\mathbb{C})$ and every Borel set $A\subseteq \mathcal{J}(f)$, we have the following functional equations

$$\left[\left(\mathcal{L}_t^n \right)^*m \right](g)=\int \mathcal{L}_t^n g dm, \hspace{1.5cm} \left[\left(\mathcal{L}_t^n \right)^*m \right](A)=\int \mathcal{L}_t^n \mathbb{1}_A dm.$$
\begin{proposition}\label{equivalence of eigenmeasure and conformal measure}
     Given a constant $c$ and a $ce^{-\Phi_t}-$conformal measure $m$ on $\mathcal{J}(f)$. The following conditions are equivalent,
     
    (a) $m(f^n(A))=c^n\int_A\exp(-S_n\Phi_t)\text{d}m$, for every $n\in \mathbb{N}$ and every Borel $A\subseteq \mathcal{J}(f)$ where $f^n$ is injective on $A$.

    (b) $m(f(A))=c\int_A\exp(-\Phi_t)\text{d}m$, for every Borel $A\subseteq \mathcal{J}(f)$ where $f$ is injective on $A$.

    (c) $\left(\mathcal{L}_t \right)^*m=cm$

    (d) $\left(\mathcal{L}_t^n \right)^*m=c^nm$ for every $n\in \mathbb{N}$.
\end{proposition}
\begin{proof}
    It is enough to show $(a) \implies (b) \implies (c) \implies (d) \implies (a).$ All implications are obvious, except for two of them. The argument for $(b) \implies(c)$ up to some modifications is very similar to the argument presented in lemma 13.6.13 in \citep[p. 465]{Munday}. We skip repeating it here. For $(d) \implies (a)$ we observe that
    \begin{align*}
        c^n\int_A \exp(-S_n\Phi_t)dm_t=m_t\left(c^n\exp(-S_n\Phi_t) \mathbb{1}_A \right)&=c^{-n}\left(\mathcal{L}_t^n \right)^*m_t\left(c^n\exp(-S_n\Phi_t) \mathbb{1}_A \right)\\
        &=\left(\mathcal{L}_t^n \right)^*m_t\left(\exp(-S_n\Phi_t) \mathbb{1}_A \right)=\int \mathcal{L}_t^n \left( \exp(-S_n\Phi_t) \mathbb{1}_A \right)dm_t\\
        &=\int \sum_{f^n(z)=w} \left( \exp \left(S_n\Phi_t(z) \right)\exp \left(-S_n\Phi_t(z) \right)\mathbb{1}_A(z)\right)dm_t(w)\\
        &=\int \sum_{f^n(z)=w} \mathbb{1}_A(z)dm_t(w)=\int \mathbb{1}_{f^n(A)}dm_t=m_t(f^n(A)).
    \end{align*}    
    
\end{proof}

\noindent We show below that a conformal measure exists for each $f\in $ BK. We start by introducing the following series for a fixed $w_0 \in \mathcal{J}(f)$,
$$S_0=\sum_{n=1}^{\infty}e^{-ns}\mathcal{L}_t^n\mathbb{1}(w_0)=\sum_{n=1}^{\infty}\exp(\log \mathcal{L}_t^n\mathbb{1}(w_0) -ns).$$
Note that $S_0$ converges for $s>P_t$ and it diverges for $s<P_t$. When $s=P_t$, we have two cases. In either case, one can find a sequence of positive real numbers $\{b_n\}$ such that $\lim_{n\to \infty}\frac{b_{n}}{b_{n+ 1}}=1$ and the modified series
$$S=\sum_{n=1}^{\infty}b_ne^{-ns}\mathcal{L}_t^n\mathbb{1}(w_0)=\sum_{n=1}^{\infty}b_n\exp(\log \mathcal{L}_t^n\mathbb{1}(w_0) -ns),$$
converges when $s>P_t$ and it diverges when $s\leq P_t$ as shown in \citep[p. 356]{MeromorphicDynamics}. Additionally, it is not hard to see that $S$ is a decreasing function of $s$ on the interval $(P_t,\infty)$. 

\noindent Now for every $s>P_t$ we define a measure 
$$\nu_s=\dfrac{1}{S}\sum_{n=1}^{\infty}b_ne^{-ns}\left(\mathcal{L}_t^n \right)^*\delta_{w_0},$$
which turns out to be a probability measure using the above functional equations. We remind readers that $B(R)=\{z \; : \; |z|>R\}\cup \{\infty \}$ and we prove the following lemma.
\begin{lemma}\label{shriking transfer operator outside diks}
    There exist $r_t, c_t >0$ such that for every $w\in \mathcal{J}(f)$ and every $R>1$, we have $\mathcal{L}_t \mathbb{1}_{B}(w)\leq \frac{c_t}{R^{r_t}}$, where $B=B(R) \cap \mathcal{J}(f)$.
\end{lemma}
\begin{proof}
Due to the fact that $|w|\geq T>0$ for all $w \in \mathcal{J}(f)$ and lemma \ref{estimate for tau derivative stronger} we find,
    \begin{align*}
        \mathcal{L}_t \mathbb{1}_{B}(w) =\sum_{f(z)=w} \left|(f_z^{-1})'(w)\right|_{\tau}^t\mathbb{1}_{B}(z)
        &=\sum_{\substack{f(z)=w \\ |z|>R}} \left|(f_z^{-1})'(w)\right|_{\tau}^t \\
        & \leq \sum_{\substack{f(z)=w \\ |z|>R}} c^{t}\dfrac{|w|^{-(1+\frac{1}{M}-\tau)t}}{|z|^{(\tau-1)t}} \\
        & \leq c^t T^{-(1+\frac{1}{M}-\tau)t} \sum_{\substack{f(z)=w \\ |z|>R}} \dfrac{1}{|z|^{(\tau-1)t}}.
    \end{align*}
    Since $t>\frac{\rho}{\tau-1}$, we can choose $\alpha$ such that $\rho<\alpha<(\tau-1)t$, then for $|z| > R$ we have 
    $$|z|^{(\tau-1)t}=|z|^{\alpha}|z|^{(\tau-1)t-\alpha}\geq |z|^{\alpha}R^{(\tau-1)t-\alpha}\geq |z|^{\alpha}R^{(\tau-1)t-\rho}.$$
    Therefore, the above estimate can be continued with,
    $$\mathcal{L}_t \mathbb{1}_{B}(w) \leq c^t T^{-(1+\frac{1}{M}-\tau)t} \sum_{\substack{f(z)=w \\ |z|>R}} \dfrac{1}{|z|^{(\tau-1)t}} \leq c^t T^{-(1+\frac{1}{M}-\tau)t} \dfrac{1}{R^{(\tau-1)t-\rho}}\sum_{\substack{f(z)=w \\ |z|>R}} \dfrac{1}{|z|^{\alpha}}.$$
    It is enough to consider $r_t=(\tau-1)t-\rho$ and $c_t=c^t T^{-(1+\frac{1}{M}-\tau)t}M_{\alpha}$.
\end{proof}
\noindent This helps us to show that the farther we get from the origin, the smaller the density of $\nu_s$. 
\begin{lemma}\label{tighetness}
    There exist $r_t, c_t >0$ such that for every $R>1$ and every $s>P_t$ we have 
    $$\nu_s \left( B(R)\cap \mathcal{J}(f) \right) \leq \frac{c_t}{R^{r_t}}.$$
\end{lemma}
\begin{proof}
First, note that for a Borel set $A\subseteq \mathcal{J}(f)$ we have,
\begin{align*}
       \nu_s \left(A \right) =\dfrac{1}{S}\sum_{n=1}^{\infty}b_ne^{-ns}\left[\left(\mathcal{L}_t^n \right)^*\delta_{w_0}\right]\left(A \right) &=\dfrac{1}{S}\sum_{n=1}^{\infty}b_ne^{-ns}\int \mathcal{L}_t^n \mathbb{1}_A d\delta_{w_0} \\
       &=\dfrac{1}{S}\sum_{n=1}^{\infty}b_ne^{-ns} \mathcal{L}_t^n \mathbb{1}_A(w_0)
       \end{align*}
This implies 
       $$\dfrac{1}{S}b_1e^{-s}\mathcal{L}_t \mathbb{1}(w_0) \leq \nu_s(\mathcal{J}(f))=1.$$
Now we set $B=B(R)\cap \mathcal{J}(f)$, then by the above estimate and lemma \ref{shriking transfer operator outside diks} we can write
    \begin{align*}
       \nu_s \left(B \right) &=\dfrac{1}{S}b_1e^{-s}\mathcal{L}_t \mathbb{1}_B(w_0) +\dfrac{1}{S}\sum_{n=2}^{\infty}b_ne^{-ns} \mathcal{L}_t^{n-1}\left(\mathcal{L}_t \mathbb{1}_B\right)(w_0) \\
       &\leq \dfrac{1}{\mathcal{L}_t \mathbb{1}(w_0)}\mathcal{L}_t \mathbb{1}_B(w_0)+ \dfrac{1}{S}\sum_{n=2}^{\infty}b_ne^{-ns} \sum_{f^{n-1}(z)=w_0}|(f^{n-1})'(z)|_{\tau}^{-t} \mathcal{L}_t \mathbb{1}_B(z) \\
       &\leq \dfrac{1}{\mathcal{L}_t \mathbb{1}(w_0)}\dfrac{c_t}{R^{r_t}} + \dfrac{1}{S}\sum_{n=2}^{\infty}b_ne^{-ns}\frac{c_t}{R^{r_t}} \sum_{f^{n-1}(z)=w_0}|(f^{n-1})'(z)|_{\tau}^{-t} \\
       &= \dfrac{1}{\mathcal{L}_t \mathbb{1}(w_0)}\dfrac{c_t}{R^{r_t}} + \dfrac{1}{S}\sum_{n=2}^{\infty}b_ne^{-ns}\frac{c_t}{R^{r_t}} \mathcal{L}_t^{n-1}\mathbb{1}(w_0) \\
       &= \dfrac{1}{\mathcal{L}_t \mathbb{1}(w_0)}\dfrac{c_t}{R^{r_t}} + \frac{c_t}{R^{r_t}}\dfrac{e^{-s}}{S}\sum_{n=2}^{\infty}\frac{b_n}{b_{n-1}}b_{n-1}e^{-(n-1)s} \mathcal{L}_t^{n-1}\mathbb{1}(w_0) \\
       &\asymp \dfrac{1}{\mathcal{L}_t \mathbb{1}(w_0)}\dfrac{c_t}{R^{r_t}} + \frac{c_t}{R^{r_t}}\dfrac{e^{-s}}{S}\sum_{n=2}^{\infty}b_{n-1}e^{-(n-1)s} \mathcal{L}_t^{n-1}\mathbb{1}(w_0) \\
       &=\dfrac{1}{\mathcal{L}_t \mathbb{1}(w_0)}\dfrac{c_t}{R^{r_t}} + \frac{c_t}{R^{r_t}}e^{-s}\nu_s(\mathcal{J}(f)) \leq \dfrac{1}{R^{r_t}}\left(\dfrac{c_t}{\mathcal{L}_t \mathbb{1}(w_0)}+c_te^{-P_t} \right).
    \end{align*}
\end{proof}
\begin{proposition}\label{conformal measure existence}
    There exists a $e^{P_t}e^{-\Phi_t}$-conformal measure $m_t$ for $f$.
\end{proposition}
\begin{proof}
    Given a sequence of numbers $\{s_j\}$ such that $s_j>P_t$ and $s_j \to P_t$, there exists a probability measure $m_t$ on $\mathcal{J}(f)$ such that a subsequence of $\{ \nu_{s_j}\}$ converges weakly to $m_t$ by Prokhorov theorem, see theorem 8.6.2, Vol II \citep[~ p. 202]{Bogachev}. Given any $K>1$, there is $N$ such that for any $n>N$
    $$ K^{-1} \leq \dfrac{b_n}{b_{n+1}}\leq K.$$
    Without loss of the results we proved so far we can assume $b_1=b_2=...=b_N=1$. Therefore, for every $n\in \mathbb{N}$ 
    $$ K^{-1} \leq \dfrac{b_n}{b_{n+1}}\leq K.$$
    By the equation,
    $$\left(\mathcal{L}_t \right)^*\nu_s =\dfrac{1}{S}\sum_{n=1}^{\infty}b_ne^{-ns}\left(\mathcal{L}_t^{n+1} \right)^*\delta_{w_0}=\dfrac{e^s}{S}\sum_{n=1}^{\infty}\frac{b_n}{b_{n+1}}b_{n+1}e^{-(n+1)s}\left(\mathcal{L}_t^{n+1} \right)^*\delta_{w_0},$$
    we find the following estimate
    $$K^{-1}\dfrac{e^s}{S}\sum_{n=1}^{\infty}b_{n+1}e^{-(n+1)s}\left(\mathcal{L}_t^{n+1} \right)^*\delta_{w_0}\leq \left(\mathcal{L}_t \right)^*\nu_s  \leq K\dfrac{e^s}{S}\sum_{n=1}^{\infty}b_{n+1}e^{-(n+1)s}\left(\mathcal{L}_t^{n+1} \right)^*\delta_{w_0}.$$
    Now by adding the term $b_1e^{-s}\left(\mathcal{L}_t\right)^*\delta_{w_0}$ to the sum, we find
    $$K^{-1}e^s\left( \nu_s - \frac{1}{S}b_1e^{-s}\left(\mathcal{L}_t\right)^*\delta_{w_0}\right)\leq \left(\mathcal{L}_t \right)^*\nu_s \leq Ke^s\left( \nu_s - \frac{1}{S}b_1e^{-s}\left(\mathcal{L}_t\right)^*\delta_{w_0}\right).$$
    Now taking the limit over the obtained subsequence of $s_{j}$, we obtain
    $$K^{-1}e^{P_t}m_t\leq \left(\mathcal{L}_t \right)^*m_t \leq Ke^{P_t}m_t.$$
    Since $K>1$ was arbitrary, we deduce $\left(\mathcal{L}_t \right)^*m_t=e^{P_t}m_t.$ Thuse, proposition \ref{equivalence of eigenmeasure and conformal measure} finishes the proof.
\end{proof}
\section{Construction of Invariant Gibbs State}

\noindent We introduce the normalized transfer operator 
 $$\hat{\mathcal{L}}_t: C_b(\mathcal{J}(f),\mathbb{C}) \to C_b(\mathcal{J}(f),\mathbb{C})$$
$$\hat{\mathcal{L}}_t=e^{-P_t}\mathcal{L}_t.$$
\begin{lemma}\label{upper bound normalized transfer operator}
    There is $L>0$ such that for every $n \in \mathbb{N}$, we have $\|\hat{\mathcal{L}}_t^n\mathbb{1}\|\leq L.$
\end{lemma}
\begin{proof}
    First, from the lemma \ref{estimate for transfer operator} it follows 
    $$\lim_{w \to \infty} \hat{\mathcal{L}}_t \varphi(w)=0, \hspace{1cm} \varphi \in C_b \left(\mathcal{J}(f), \mathbb{C} \right).$$
    This implies $\| \hat{\mathcal{L}}_t^n \mathbb{1} \|=\| \hat{\mathcal{L}}_t \hat{\mathcal{L}}_t^{n-1} \mathbb{1} \|=\hat{\mathcal{L}}_t^n \mathbb{1}(w_n)$ for some $w_n\in \mathcal{J}(f)$.
    We choose $R>0$ such that $\hat{\mathcal{L}}_t \mathbb{1}(w) \leq 1$ when $|w|\geq R$.
    By lemma \ref{distortion of transfer operator} there exist $c>0, \;K_R>1$ such that for every $w\in B=\mathcal{J}(f)\cap D(0,R)$, we have
    \begin{equation}\label{in proof 1}
        m_t(B)c^{-1}K_R^{-t}\hat{\mathcal{L}}_t^n \mathbb{1}(w)\leq \int_B \hat{\mathcal{L}}_t^n \mathbb{1}dm_t\leq \int \hat{\mathcal{L}}_t^n \mathbb{1}dm_t=1,
    \end{equation}
    where the last equality holds due to $\left(\hat{\mathcal{L}}_t \right)^*m_t=m_t$ as shown in proposition \ref{conformal measure existence}. For $L=\max \left(1, cK_R ^t\frac{1}{m_t(B)}\right)$, we then have $\|\hat{\mathcal{L}}_t \mathbb{1} \| \leq L$. Given $\|\hat{\mathcal{L}}_t^k\mathbb{1}\|\leq L$. If $|w_{k+1}|\geq R$, then 
$$\|\hat{\mathcal{L}}_t^{k+1}\mathbb{1}\|=\hat{\mathcal{L}}_t^{k+1}\mathbb{1}(w_{k+1})=\hat{\mathcal{L}}_t \left(\hat{\mathcal{L}}_t^{k}\mathbb{1}\right)(w_{k+1})\leq \|\hat{\mathcal{L}}_t^k\mathbb{1}\|\leq L.$$
Otherwise, we have $w_{k+1}\in B$ and the estimate (\ref{in proof 1}) also yields $\|\hat{\mathcal{L}}_t^{k+1}\mathbb{1}\|\leq L.$ Therefore by induction, the proof is completed. 
    
\end{proof}
\begin{lemma}\label{lower bound normalized transfer operator}
    For every $R>0$, there exists $l_R>0$ such that for every  $w \in \mathcal{J}(f)\cap D(0,R)$ and every $n\in \mathbb{N}$, we have $l_R \leq \hat{\mathcal{L}}_t^n\mathbb{1}(w)$.
\end{lemma}
\begin{proof}
    By lemma \ref{upper bound normalized transfer operator} we know there exists $L>0$ such that $\|\hat{\mathcal{L}}_t^n\mathbb{1}\|\leq L$ for every $n\in \mathbb{N}$. Following lemma \ref{tighetness} we can choose $R>0$ such that $m_t(B(R))\leq \frac{1}{4L}$. Therefore,
    $$1=\int \hat{\mathcal{L}}_t^n \mathbb{1}dm_t =\int _{D(0,R)}\hat{\mathcal{L}}_t^n \mathbb{1}dm_t+\int_{B(R)}\hat{\mathcal{L}}_t^n \mathbb{1}dm_t\leq \int _{D(0,R)}\hat{\mathcal{L}}_t^n \mathbb{1}dm_t+\frac{1}{4}.$$
    Hence, there should be $w_n \in \mathcal{J}(f) \cap D(0,R)$ such that $\hat{\mathcal{L}}_t^n \mathbb{1}(w_n)\geq \frac{3}{4}$. Now applying lemma \ref{distortion of transfer operator} we find $l_R>0$ such that $\frac{3}{4}l_R \leq l_R \hat{\mathcal{L}}_t^n \mathbb{1}(w_n)\leq \hat{\mathcal{L}}_t^n \mathbb{1}(w)$ for every $w \in \mathcal{J}(f) \cap D(0,R)$.
\end{proof}

\noindent Now we can give a better description of the pressure.
\begin{proposition}
    $P_t= \lim_{n \to \infty}\frac{1}{n}\log \mathcal{L}_t^n\mathbb{1}(w),$ for every $w \in \mathcal{J}(f).$
\end{proposition}
\begin{proof}
    This follows from lemmas \ref{upper bound normalized transfer operator} and \ref{lower bound normalized transfer operator}. Given $w\in \mathcal{J}(f)$ and $R>0$ such that $|w|<R$ then $l_R \leq \hat{\mathcal{L}}_t^n \mathbb{1}(w)\leq L$, i.e. $l_r \leq e^{-nP_t}{\mathcal{L}}_t^n \mathbb{1}(w)\leq L$.
\end{proof}

\noindent Next, we briefly explain how we obtain a $f-$invariant measure. For more details we refer the readers to \citep[p. 451, 453]{Munday}. We emphasize that the argument in \citep[p. 451, 453]{Munday} is under the expanding condition of $f$ which we do not have it. However, one can easily check that the proofs of what we require from \citep[p. 451, 453]{Munday} go through with some modifications.

\noindent Since $(\hat{\mathcal{L}}_t)^*m_t=m_t$ as shown in proposition \ref{conformal measure existence}, the conformal measure $m_t$ is quasi-$f-$invariant by proposition 13.5.2 in \citep[p. 453]{Munday}. Alternatively, lemma \ref{quasi-invariant} implies $m_t$ is quasi $f-$invariant with no use of eigenmeasure. Therefore, as long as we have 
\begin{itemize}
    \item[(1)] $\hat{\mathcal{L}}_th=h$ \;\;\;\; for some $h \in C_b(\mathcal{J}(f),\mathbb{R})$,
    \item[(2)] $\mathcal{L}_{m_t}=\hat{\mathcal{L}}_t$ \;\;\; a.e. on $\mathcal{J}(f)$.
\end{itemize} 
the Borel probability measure $\mu_t$ given by 
\begin{equation}\label{invariant measure}
    \mu_t(A)=\int_A h\text{d}m_t, \hspace{.5cm} \text{Borel set $A \subseteq \mathcal{J}(f)$}
\end{equation}
would be $f-$invariant by theorem 13.4.1 in \citep[p. 451]{Munday}. Item (2) is deducted from propositions 13.4.2 \& 13.5.2 in \citep[p. 451, 453]{Munday}. Therefore, we are left to show only item (1). Before showing item (1) in lemma \ref{fixed point of nomalized operator}, we express the following proposition which follows from the above explanation.
\begin{proposition}\label{existence Gibbs}
    There exists a $f-$invariant Borel probability measure $\mu_t$ which is absolutely continuous with respect to the conformal measure $m_t$.
\end{proposition}
\begin{lemma}\label{fixed point of nomalized operator}
    The normalized operator $\hat{\mathcal{L}}_t$ has a fixed point $h$.
\end{lemma}
\begin{proof}
    We define a sequence of functions below,
    $$h_n(w)=\dfrac{1}{m}\sum_{k=1}^m \hat{\mathcal{L}}_t^n \mathbb{1}(w), \; \; \; w \in \mathcal{J}(f).$$
    Note that for every $R>0$ we have $l_R \leq h_n(w)\leq L$ for every $w\in \mathcal{J}(f)\cap D(0,R)$ by lemmas \ref{upper bound normalized transfer operator} and \ref{lower bound normalized transfer operator}. As $L$ is independent of $R$, so $(h_n)$ is uniformly bounded. Furthermore, lemma \ref{distortion of transfer operator 0} implies that $(h_n)$ is an equicontinuous family of functions. Hence, by Arzela-Ascoli's theorem we obtain a subsequence $(h_{n_j})$ of $(h_n)$ such that $h_{n_j} \rightarrow  h \in C_b(\mathcal{J}(f),\mathbb{R})$ uniformely on compact sets of $\mathcal{J}(f)$. Next, we show $\hat{\mathcal{L}}_t h_{n_j} \to \hat{\mathcal{L}}_t h$ (pointwise convergence). Given $\epsilon>0$, for every $w\in \mathcal{J}(f)$ there exists $R>0$ such that  $\sum_{\substack{f(z)=w \\ |z|>R}} \exp(\Phi_t(z)-P_t) < \epsilon,$
    as a tail of $\hat{\mathcal{L}}_t \mathbb{1}(w)$. Also, there exists $j_0$ such that $|h_{n_j}(z)-h(z)| <\epsilon$ for all $j\geq j_0$.
    Hence,
    \begin{align*}
         & \; \; \; \;  \left| \hat{\mathcal{L}}_t h_{n_j}(w) - \hat{\mathcal{L}}_t h(w)\right|\\
        &\leq \left|\sum_{\substack{f(z)=w \\ |z|\leq R}} \left( h_{n_j}(z)-h(z)\right)\exp\left(\Phi_t(z)-P_t  \right) \right|+\left|\sum_{\substack{f(z)=w \\ |z|>R}} \left( h_{n_j}(z)-h(z)\right)\exp\left(\Phi_t(z)-P_t  \right) \right|\\
        &\leq \epsilon\hat{\mathcal{L}}_t \mathbb{1}(w)+2L\epsilon\leq 3L\epsilon,
    \end{align*}
    for all $j\geq j_0$.
    But we know $\hat{\mathcal{L}}_t h_{n_j}=\frac{1}{n_j}\left(\sum_{k=1}^{n_j}\hat{\mathcal{L}}_t^k\mathbb{1} -\hat{\mathcal{L}}_t^k\mathbb{1}+\hat{\mathcal{L}}_t^{n_j+1}\mathbb{1}\right) \to h$. Thus, $\hat{\mathcal{L}}_th=h$.
\end{proof}
\begin{lemma}\label{derivative}
     $h=\frac{d\mu_t}{dm_t}$ has the following properties,
    \begin{itemize}
        \item[(a)] For every $R>0$ there exist positive constants $L,l_R$ such that $l_R \leq h(w) \leq L$ for every $w\in \mathcal{J}(f)\cap D(0,R)$. The constants $L, l_R$ are the same as the constants obtained in lemmas \ref{upper bound normalized transfer operator} and \ref{lower bound normalized transfer operator}.
        \item[(b)] There exists $c_t>0$ such that for every $w \in \mathcal{J}(f)$ we have $h(w)\leq c_t|w|^{-(1+\frac{1}{M}-\tau)t}$.
    \end{itemize}
\end{lemma}
\begin{proof}
    Item (a) follows from the proof of lemma \ref{fixed point of nomalized operator} that for every $R>0$ we have $l_R \leq h_n(w)\leq L$ for every $w\in \mathcal{J}(f)\cap D(0,R)$. Item (b) follows from lemma \ref{estimate for transfer operator} (that $h$ is a fixed point of $\hat{\mathcal{L}}_t$) and lemma \ref{fixed point of nomalized operator}.
\end{proof}
\begin{proposition}\label{equivalence of measures}
    The invariant measure $\mu_t$ is equivalent to the conformal measure $m_t$.
\end{proposition}
\begin{proof}
    This follows from lemma \ref{derivative} item (a).
\end{proof}
\section{Uniqueness of Conformal Measure and Ergodicity}
\begin{lemma}\label{quasi-invariant}
    Given a constant $c$ and a $ce^{-\Phi_t}-$conformal measure $m$ for $f$. Then the following items hold.
    \begin{itemize}
        \item[(a)] For every $R>0$, there exists $c_R>0$ such that for every Borel set  $B \subseteq \mathcal{J}(f) \cap B(R)$,
    $$m(f^{-1}(B))\leq c_R m(B),$$
    where $c_R \to 0$, as $R\to \infty$.
        \item[(b)]  $m$ is quasi-invariant, i.e. $m(f^{-1}(B))=0$ when $m(B)=0$.
        \item[(c)] There exist $R>0$ and $0<c_R<1$ such that for every $n \in \mathbb{N}$ and every Borel set  $B \subseteq \mathcal{J}(f) \cap B(R)$,
        $$m \left( \cap_{i=0}^{n}f^{-i}(B) \right) \leq c_R^nm(B).$$
        \item[(d)]  For every $w_1\in \mathcal{J}(f)$, every $z\in f^{-n}(\{w_1\})$ and every Borel set  $B \subseteq \mathcal{J}(f) \cap D(w_1,\delta)$,
    $$m(f_z^{-n}(B))= c^{-n}\int_B \exp\left(S_n\Phi_t(f_z^{-n}(w))\right)dm(w).$$
    \end{itemize}
\end{lemma}
\begin{proof}
     Given $f(z)=w\in \mathcal{J}(f)$ and a Borel set $B\subseteq D(w,\delta) \cap \mathcal{J}(f)$, where $\delta$ was given in the equation (\ref{delta}). By lemma \ref{exp Ionescu S_n} one can find $K_t>1$ such that for all $w_1\in B,$
     $$\exp\left(\Phi_t\left(f^{-1}(w_1)\right)\right)=\left|(f_z^{-1})'(w_1)\right|_{\tau}^{t}\leq K_t \left|(f_z^{-1})'(w)\right|_{\tau}^{t}.$$
     Therefore, by lemma \ref{estimate for tau derivative stronger} and the fact that $f: f^{-1}_z(B) \to B$ is bijective we obtain,
     $$m(B)=m(f(f^{-1}_z(B)))=\int_{f_z^{-1}(B)}ce^{-\Phi_t}\text{d}m \geq cK_t^{-1}(d^{t})^{-1}\dfrac{|w|^{(1+\frac{1}{M}-\tau)t}}{|z|^{-(\tau-1)t}} m(f_z^{-1}(B)).$$
     This means there exists $b_t>0$ such that 
     $$m(f_z^{-1}(B)) \leq b_t \dfrac{|w|^{-(1+\frac{1}{M}-\tau)t}}{|z|^{(\tau-1)t}} m(B).$$
     and therefore by the estimate (\ref{Borel theorem 2}) we have,
     \begin{align*}
         m(f^{-1}(B))=\sum_{f(z)=w}m(f_z^{-1}(B)) &\leq b_tm(B)|w|^{-(1+\frac{1}{M}-\tau)t}\sum_{f(z)=w}\dfrac{1}{|z|^{(\tau-1)t}}\\
         &\leq  b_tm(B)|w|^{-(1+\frac{1}{M}-\tau)t} M_{(\tau-1)t}
     \end{align*}
     Given $R>0$ and a Borel set $B\subseteq \mathcal{J}(f) \cap B(R)$, one can find a countable set $\{w_i\}_i \subseteq B$ such that 
     $B \subseteq \cup_i D(w_i,\delta)$. Inductively, we can obtain $B_i\subseteq D(w_i,\delta)$ such that $B=\cup_i B_i$ where $B_{i}$'s are disjoint. Then the above estimate gives,
     \begin{align*}
    m\left(f^{-1}(B)\right) = \sum_i m\left(f^{-1}(B_{i})\right) &\leq \sum_i b_t m(B_i)|w_{i}|^{-(1+\frac{1}{M}-\tau)t}M_{(\tau-1)t}\\
    & \leq \sum_i b_tm(B_i)R^{-(1+\frac{1}{M}-\tau)t}M_{(\tau-1)t}\\
    & \leq b_tR^{-(1+\frac{1}{M}-\tau)t}M_{(\tau-1)t}m(B).
     \end{align*}
     This proves item (a).\\
     Item (b) follows from item (a) with regards to (\ref{T}).\\
     For item (c), first one can choose $R>0$ large enough such that $c_R<1$ by item (a). Then $m(B\cap f^{-1}(B))\leq m(f^{-1}(B))\leq c_Rm(B)$. Therefore, by the following estimate
     $$m\left( \cap_{i=0}^nf^{-i}(B)\right) \leq m\left( \cap_{i=1}^nf^{-i}(B)\right)=m\left( f^{-1}\left(\cap_{i=0}^{n-1}f^{-i}(B)\right)\right),$$
     one can use induction to establish item (c).\\
     Finally, item (d) follows from proposition \ref{equivalence of eigenmeasure and conformal measure} item (d).
\end{proof}
\begin{lemma}\label{limit inf bounded}
    Given a constant $c$ and a $ce^{-\Phi_t}-$conformal measure $m$ for $f$. There exists $R_t>0$ such that 
    $$\liminf_{n \to \infty}|f^n(z)|\leq R_t \hspace{.7cm} m-\text{a.e. \;} z \in \mathcal{J}(f).$$
\end{lemma}
\begin{proof}
    We choose $R>0$ from item (c) in lemma \ref{quasi-invariant}. Note that we have the following inclusion, 
    $$\{z \in \mathcal{J}(f) \; : \; \liminf_{n \to \infty} |f^n(z)| >R\} \subseteq \cup_{i=0}^{\infty} \cap_{j=i}^{\infty}f^{-j}(B)=\cup_{i=0}^{\infty}f^{-i}\left( \cap_{j=0}^{\infty}f^{-j}(B)\right),$$
    where $B=\mathcal{J}(f)\cap B(R)$. Clearly, $m\left( \cap_{j=0}^{\infty}f^{-j}(B)\right)=0$ by the same item (c). Now we apply item (b) in lemma \ref{quasi-invariant} to find  $m\left(f^{-i} \left( \cap_{j=0}^{\infty}f^{-j}(B)\right)\right)=0$ for each $i$. This finishes the proof.
\end{proof}
\noindent We now prove an inequality resembling Gibbs property for every conformal measure. First, we consider the set $\mathcal{J}_{r,N} \subseteq \mathcal{J}_r(f)$ consisting of $z$ where the orbit of $z$ under $f$ has a limit point in $D(0, N)$, where $N>0$. Note that when $N=R_t$ from the previous lemma, then clearly $m(\mathcal{J}_{r,N})=1.$
\begin{lemma}\label{Gibbs}
    Given a constant $c$ and a $ce^{-\Phi_t}-$conformal measure $m$ for $f$. For every $N>0$, there exists $K_N>1$ such that for every $z \in \mathcal{J}_{r,N}$ we have
    $$K_N^{-1}c^{-n_k}\exp(S_{n_k}\Phi_t(z)) \leq m\left(\mathcal{J}(f)\cap D\left(z,\frac{\delta}{4} |(f^{n_k})'(z)|^{-1}\right) \right) \leq K_N c^{-n_k}\exp(S_{n_k}\Phi_t(z)),$$
    for every $k\in\mathbb{N}$.
\end{lemma}
\begin{proof}
    By Koebe's distortion theorem for every $n\in \mathbb{N}$, we have
    \begin{equation}\label{Koebe inclusion}
        D(f^n(z),\frac{\delta}{16}) \subseteq f^n \left(D\left(z,\frac{\delta}{4}|(f^n)'(z)|^{-1}\right) \right), \hspace{.3cm} D(z,\frac{\delta}{4}|(f^n)'(z)|^{-1}) \subseteq f_z^{-n} \left(D\left(f^n(z),\delta \right) \right).
    \end{equation}
    Let $z\in \mathcal{J}_{r,N}$,  there is a subsequence $n_k$ such that $f^{n_k}(z) \to y \in D(0,N)$ as $k \to \infty$. By the Rippon-Stallard (\ref{Rippon-Stallard-modified}) estimate, we find $|(f^{n_k})'(z)|^{-1} < 1$ for all large enough $k$. Now the latter inclusion in (\ref{Koebe inclusion}) combined with lemma \ref{quasi-invariant} item (d) and lemma \ref{exp Ionescu S_n} gives the right-hand inequality, while the former one implies
    $$f_z^{-n_k}\left(D(y,\frac{\delta}{32})\right) \subseteq D\left(z,\frac{\delta}{4}|(f^{n_k})'(z)|^{-1}\right) ,$$
    for all large enough $k$. This is enough to obtain the left-hand inequality by lemma \ref{quasi-invariant} item (d) and lemma \ref{exp Ionescu S_n}.
\end{proof}
\begin{proposition}\label{unique}
    Given a constant $c$ and a $ce^{-\Phi_t}-$conformal measure $m$ for $f$, then $c=e^{P_t}$. Also, $m$ and $m_t$ are equivalent.
\end{proposition}
\begin{proof}
    Let $N=R_t$ from lemma \ref{Gibbs} and lemma \ref{limit inf bounded} and consider a Borel subset $B$ of $\mathcal{J}_r(f)$. Also, let $B'=B\cap \mathcal{J}_{r,N}$. Suppose $\epsilon>0$, then there exists an open set $O$ containing $B'$ such that $m_t(O \setminus B')<\epsilon$ due to regularity of $m_t$. Observe that $\{ D(z,r_z) \}_{z \in O}$ is a Besicovic cover (see \citep[p. 103]{DiBenedetto}) for $B'$ where $r_z=\frac{\delta}{4} |(f^{n_k})'(z)|^{-1}$ for some $k$ large enough from lemma \ref{Gibbs} such that $n_k>\epsilon^{-1}$. Hence, by Besicovic theorem (see \citep[p. 103]{DiBenedetto}) one can find a countable subcover $\{ D(z_i,r_{z_i}) \}_i$ for $B'$ where each $z\in B'$ appears in at most $C$ different members of the subcover, where $C$ is independent of $B$ and the cover. Now we assume $c>e^{P_t}$ and  apply lemma \ref{Gibbs} for $m$ and $m_t$ to find $K_N>1$ and $K'_N>1$ such that,
    \begin{align*}
        m(B)=m(B\cap \mathcal{J}_{r,N})=m(B') & \leq \sum_{i=1}^{\infty}m \left(D(z_i,r_{z_i}) \right)\\
        &\leq K_N \sum_{i=1}^{\infty} c^{-n_k}\exp (S_{n_k}\Phi_t(z_i))\\
        &=K_N\sum_{i=1}^{\infty} c^{-n_k}e^{P_tn_k}e^{-P_tn_k}\exp (S_{n_k}\Phi_t(z_i))\\
        & \leq K_NK'_N\sum_{i=1}^{\infty}(c^{-1}e^{P_t})^{n_k}m_t \left( D(z_i,r_{z_i}) \right)\\
        & \leq K_nK'_N (c^{-1}e^{P_t})^{1/\epsilon}Cm_t\left(\cup_{i=1}^{\infty}D(z_i,r_{z_i}) \right)\\
        & \leq K_nK'_N (c^{-1}e^{P_t})^{1/\epsilon}C \left( \epsilon+m_t(B') \right).
    \end{align*}
    Since $\epsilon>0$ is arbitrary, so $m(B)=0$. In case $c<e^{P_t}$, we also find $m_t(B)=0$ by exchanging $m$ and $m_t$ in the above estimate. In either case, the measure of $B=\mathcal{J}(f)$ is $0$, which is a contradiction. Thus, we must have $c=e^{P_t}$. Also, the above estimate gives
    $$\frac{1}{M}m_t(B) \leq m(B) \leq M m_t(B),$$
    where $M=CK_NK'_N$.
    This implies that $m$ and $m_t$ are equivalent 
\end{proof}
\begin{proposition}\label{ergodic}
    Every $e^{P_t}e^{-\Phi_t}-$ conformal measure $m$ is ergodic.
\end{proposition}
\begin{proof}
    Given $f^{-1}(C)=C$ where $C \subseteq \mathcal{J}(f)$ is a Borel set with $m(C)\neq0,1$. One obtains two conditional measures,
    $$m_{1}(B)=\dfrac{m(C\cap B)}{m(C)},  \hspace{.7cm} m_{2}(B)=\dfrac{m(C^c\cap B)}{m(C^c)}, \hspace{.3cm} C^c=\mathcal{J}(f)\setminus C,$$
    each of which $e^{P_t}e^{-\Phi_t}-$ conformal measure for $f$ while they are mutually singular. This contradicts the previous proposition \ref{unique}.
\end{proof}
\begin{theorem}\label{main theorem}
    Let $f \in $ BK, i.e. let $f$ be a transcendental meromorphic function of B--class with finite Nevanlinna order $\rho$ where $\infty$ is not an asymptotic value of $f$ and that there exists $M \in \mathbb{N}$ such that the multiplicity of all poles, except possibly finitely many, is at most $M$. For the geometric potential $\Phi_t(z)=t\log |f'(z)|^{-1}_{\tau}$ where $z\in \mathcal{J}(f)$ $, 1<\tau<1+\frac{1}{M}$ and $t> \frac{\rho}{\tau -1}$ the following hold,
    \begin{itemize}
        \item[(a)] There exits a unique $e^{P_t}e^{-\Phi_t}$-conformal measure $m_t$, where $P_t$ is the topological pressure.
        \item[(b)] There exists a unique Gibbs state $\mu_t$ that is $f-$invariant and equivalent to $m_t$.
        \item[(c)] $m_t$ and $\mu_t$ are ergodic with support on the radial Julia set.  
    \end{itemize}
\end{theorem}
\begin{proof}
    First, we show item (b). The existence of $\mu_t$ was established in proposition \ref{existence Gibbs}, the equivalence was established in proposition \ref{equivalence of measures}. One can use lemma \ref{quasi-invariant} item (d), lemma \ref{derivative} item (a), lemma \ref{exp Ionescu S_n}, proposition \ref{equivalence of measures} and the inclusion (\ref{Koebe inclusion}) to see that $\mu_t$ is a Gibbs state. It remains to show the uniqueness. It is clear that the ergodicity of $m_t$ (proved in proposition \ref{ergodic}) carries over to $\mu_t$. Therefore, every two invariant Gibbs states that are equivalent to $m_t$ must be identical.

\noindent    Next, we show item (a). The existence was established in proposition \ref{conformal measure existence}. For uniquness, given $m_1$, $m_2$ two $e^{P_t}e^{-\Phi_t}$-conformal measures and $\mu_1$, $\mu_2$ their corresponding invaraint Gibbs states. We know $m_1$ and $m_2$ are equivalent by proposition \ref{unique}. Therefore, $\mu_1$, $\mu_2$ are equivalent by item (b) and so they must be identical. This implies 
    $$dm_1=h^{-1}d\mu_1=h^{-1}d\mu_2=dm_2.$$

\noindent    Finally, the ergodicity for item (c) was established in the proof of item (a) above. Also, the support argument was discussed in lemma \ref{limit inf bounded}.
\end{proof}

\bibliographystyle{alphaurl}
\bibliography{main.bib}
\end{document}